\documentclass[a4paper,11pt,oneside]{article}
\usepackage{graphicx}
\usepackage{amscd}
\usepackage{amsmath}
\usepackage{caption}
\usepackage{amsfonts}
\usepackage{amssymb}
\usepackage{amsthm}
\usepackage{mathrsfs}
\usepackage{multicol}
\usepackage{color}
\usepackage[english]{babel}
\usepackage[T1]{fontenc}
\usepackage{textcomp}
\usepackage[utf8]{inputenc}
\usepackage{enumitem}
\usepackage{indentfirst}

\newcommand{\N}{\mathbb N}

\newcommand{\RR}{{{\rm I} \kern -.15em {\rm R} }}
\usepackage[pagewise]{lineno}

\begin{document}
	\theoremstyle{plain} \newtheorem{thm}{Theorem}[section] \newtheorem{cor}[thm]{Corollary} \newtheorem{lem}[thm]{Lemma} \newtheorem{prop}[thm]{Proposition} \theoremstyle{definition} \newtheorem{defn}{Definition}[section] 
	
	\newtheorem{oss}[thm]{Remark}
	\newtheorem{ex}{Example}[section]
	\newtheorem{lemma}{Lemma}[section]
	\title{A note on the Cucker-Smale model with time delay and communication failures}
	\author{Elisa Continelli\footnote{Email: elisa.continelli@math.unipd.it. } \\Dipartimento di Matematica "Tullio Levi-Civita",\\
		Universit\`{a} degli Studi di Padova,\\
		Via Trieste, 63, 35121, Padova, Italy}

	\maketitle

	\begin{abstract}
		In this paper, we deal with a Cucker-Smale model with time-dependent time delay and communication failures. Namely, we investigate the situation in which the agents involved in a flocking process can possibly suspend the exchange of information among themselves at some times. Under a so-called Persistence Excitation Condition, we establish the exponential flocking for the considered model. The exponential decay estimates on the velocity diameters we obtain are independent of the number of agents, which is convenient especially when the number of agents becomes too large. Exponential decay estimates with decay rate depending on the number of agents appear instead in the case of pair-dependent time delays and non-universal interaction, i.e. not all the agents are able to exchange information among themselves. In this paper, we rather consider a universal interaction and the time delay functions do not depend on the pair of agents. The analysis is then extended to a model with distributed time delay, namely the time delay is not pointwise but the agents are influenced by the information received from the other components of the system in a certain time interval.
	\end{abstract}

	\providecommand{\keywords}[1]{\textbf{Keywords:} #1}
	\keywords{flocking models; communication failures; time delay; universal interaction.}
	
	\vspace{5 mm}
	
	\section{Introduction}
	The study of multiagent systems is a widely exploited topic in the recent literature, by virtue of their application to several scientific fields, among them biology \cite{Cama, CS1}, economics \cite{Marsan}, robotics \cite{Bullo, Jad}, control theory \cite{Aydogdu, Borzi, WCB, PRT, Piccoli}, social sciences \cite{Bellomo, Campi, CF}. 
	
	Among them, there is the Hegselmann-Krause opinion formation model, formulated by Hegselmann and Krause in 2002. In the pioneering work \cite{HK} due to Hegselmann and Krause, the approach to consensus for a group of interacting individuals was investigated. In 2007, the second-order version of the Hegselmann-Krause model was introduced by Cucker and Smale in \cite{CS1} for the description of flocking phenomena (for instance, flocking of birds, schooling of fish or swarming of bacteria). Since then, several generalizations have been proposed (see, among others, \cite{Posneg,CFT2,CFT}). Typically, for solutions to the two aforementioned models, their asymptotic behavior (consensus for the first-order model or flocking for the second-order model) is analyzed. In particular, solutions to the classical formulations of the Hegselmann-Krause model and the Cucker-Smale model achieve the asymptotic consensus due to simmetry reasons.
	
	Later on, to make these models more suitable for describing real life phenomena, time delay effects have been introduced in such systems. The presence of time delays, even arbitrarily small, destroys the typical features of the original models. In particular, the symmetry is broken and it is not possible to argue as in the undelayed case, deducing the asymptotic consensus via simmmetry properties. So, time delays make the problems more difficult to deal with.
	
	The asymptotic behavior of solutions to first and second-order Cucker-Smale models in presence of a time delay (delay that can be constant or, more realistically, dependent of time), has been developed in many works \cite{LW, CH, CL, CP, PT, HM, DH, Lu, CPP, P, H}. Most of them require a smallness condition on the time delay size in order to prove the asymptotic consensus for the Hegselmann-Krause model and its second-order version. However, in the recent work \cite{Cartabia}, the asymptotic flocking for the Cucker-Smale model with constant time delay has been established without requiring any smallness assumptions on the time delay size. Upper bounds on the time delay size are not required either in \cite{H3,PR}. The result in \cite{Cartabia} has been then extended in \cite{Cont, ContPign}, where the exponential consensus for both the first and second-order model in presence of time-dependent time delays have been proven  without assuming any restrictions on the time delay size. Moreover, in \cite{Cont,ContPign}, no monotonicity properties are required on the influence function. Concerning restrictions on the time delay size, an upper bound still remains necessary for models including self-delay (see \cite{Haskovec_JMAA}). 
	
	A growing interest in these last years has been devoted to the analysis of first and second-order Cucker-Smale models under communication failures (see \cite{Caravenna, Rossi, AnconaRossi, Bonnet} for the undelayed case and \cite{CicoContPi,ContPi} for the delayed case). Namely, the agents involved in the opinion formation or flocking process can suspend their interactions at some times. This is modeled by including in the formulation of first and second-order Cucker-Smale models pair-dependent weight functions that can possibly degenerate at certain times. Definitely, the lack of connection among the system's agents prevents the opinion or velocity alignment. Then, it is important to find conditions ensuring that the system reaches the asymptotic consensus despite the lacking interaction among the agents.
	
	In \cite{CicoContPi}, a very general setting is investigated since the possible lack of connection among the agents is addressed by considering pair and time-dependent time delays and a non-universal interaction, namely there could particles that are never able to communicate. Non-universal interaction among the agents is also considered in \cite{CCCP,COP,CicolaniPignotti}. Although the result established in \cite{CicoContPi} for the first-order model is more general with respect to the one in \cite{ContPi}, where the time delay and weight functions are not pair-dependent and the interaction is universal, the decay rate in the exponential decay estimate provided in \cite{CicoContPi} depends on the number of agents, what did not happen in \cite{ContPi}. Finding decay rates that do not depend on the number of agents is rather important, especially when the number of agents becomes too large. Indeed, in \cite{ContPi} the fact that the decay rate is independent of the number of agents allows to extend the exponential consensus result to the associated PDE model, obtained as the mean field limit of the particle system when $N\to \infty$ (cf. \cite{Canizo,HaTadmor}), by using the same techniques employed in \cite{PPpreprint,CPP}. Let us also mention \cite{CH} for an exponential flocking result for the kinetic model associated with the Cucker-Smale system in presence of time delays.

	In this paper, we extend the result in \cite{ContPi} to the second-order model, namely we establish exponential flocking for the Cucker-Smale model with possible loss of connection and time-dependent time delay, obtaining a decay rate that does not depend on the number of agents. This complements the analysis in \cite{CicoContPi} for the second-order model. Indeed, the flocking result we will prove shows that, in the case of universal interaction and not pair-dependent weight and time delay functions, the decay rate provided in \cite{CicoContPi} can be chosen independently of $N$. Decay estimates independent of the number of agents are needed to extend the flocking result to the associated continuum model, obtained as mean field limit of the particle system when $N$ goes to infinity (cf. \cite{CH}).

	The rest of the paper is organized as follows. In Section \ref{timedelay} we formulate our model. We give some definitions, in particular the definition of position and velocity diameters and asymptotic flocking for solutions to the Cucker-Smale model. Then, we state our main result, ensuring that solutions to the Cucker-Smale model with time delay and communication failures exhibit exponential flocking with decay rate independent of the number of agents. In Section \ref{main}, we give some preliminary results that will be used to prove the exponential flocking. We find suitable estimates on the velocity diameters that do not depend on the number of agents and we prove our flocking result. Finally, in Section \ref{distributed} the analysis is extended to a model with distributed time delay, namely each agent changes its velocity depending on  the information received from the other	agents on a certain time interval (cf. \cite{P,CP}). 

	\section{Problem formulation}\label{timedelay}
	Consider a finite set of $N\in\N$ agents, with $N\geq 2 $. Let $x_{i}(t)\in \RR^d$ and $v_{i}(t)\in \RR^d$ denote the position and the velocity of the $i$-th particle at time $t$, respectively. We shall denote with $\lvert\cdot \rvert$ and $\langle\cdot,\cdot\rangle$ the usual norm and scalar product on $\RR^{d}$, respectively. Moreover, let us denote with $\N_0:=\N\cup\{0\}.$ We assume that the agents interact among themselves according to the following Cucker-Smale type model:
	\begin{equation}\label{onoff}
		\begin{cases}
			\frac{d}{dt}x_{i}(t)=v_{i}(t),\quad &t>0, \,\,\forall i=1,\dots,N,\\\frac{d}{dt}v_{i}(t)=\underset{j:j\neq i}{\sum}\alpha(t) b_{ij}(t)(x_{j}(t-\tau(t))-x_{i}(t)),\quad &t>0,	\,\, \forall i=1,\dots,N,
		\end{cases}
	\end{equation}
	where the time delay function $\tau:[0,+\infty)\rightarrow[0,+\infty)$ is continuous and satisfies 
	\begin{equation}\label{taubounded}
		0\leq \tau(t)\leq \bar\tau,\quad \forall t\geq 0,
	\end{equation}
	for some positive constant $\bar\tau$.
	\\Here, the communication rates $b_{ij}$ are of the form
	\begin{equation}\label{weight}
		b_{ij}(t):=\frac{1}{N-1}\psi(\lvert x_{i}(t)-x_{j}(t-\tau(t))\rvert), \quad\forall t>0,\, \forall i,j=1,\dots,N,
	\end{equation}
	where the influence function $\psi:\RR\rightarrow \RR$ is positive, bounded and continuous and  
	\begin{equation}\label{K}
		K:=\lVert \psi\rVert_{\infty}.
	\end{equation} 
	The weight function $\alpha:[0,+\infty)\rightarrow [0,1]$ is a $\mathcal{L}^1$-measurable function satisfying the following Persistence Excitation Condition (cf. \cite{AnconaRossi, CicoContPi, Bonnet}):
	\begin{itemize}
		\item [\bf (PE)]
		there exist two positive constants $T$ and $\tilde{\alpha}$ such that
		\begin{equation}\label{PE}
			\int_{t}^{t+T}\alpha(s)ds\geq \tilde{\alpha},\quad \forall t\geq 0,
		\end{equation}
		Without loss of generality, we can assume that $T\geq \bar\tau$. 
	\end{itemize}
	Let us note that \eqref{PE} becomes relevant when $T$ is large and $\tilde{\alpha}$ is small. In this case, the agents could possibly suspend their interaction for long enough.  We also point out that, in the case in which $\alpha(t)=1$, for a.e. $t\geq 0$, i.e. in the case in which the agents do not interrupt their exchange of information, the condition \eqref{PE} is of course satisfied.
	
	Due to the presence of time delays, the initial conditions are functions defined in the negative time interval $[-\bar\tau, 0].$ The initial conditions
	\begin{equation}\label{incond}
		x_{i}(s)=x^{0}_{i}(s),\quad v_{i}(s)=v^{0}_{i}(s), \quad  \forall s\in [-\bar\tau,0],\,\forall i=1,\dots,N,
	\end{equation}
	are assumed to be continuous functions.
	
	For existence result related to the delayed system \eqref{onoff}, we refer to \cite{Halanay, HL}. Here, we are interested in investigating the asymptotic behavior of solutions to \eqref{onoff}.
	
Now, we give the rigorous definition of asymptotic flocking for solutions to the Cucker-Smale model \eqref{onoff}. To this aim, we define the space diameter $d_X(\cdot)$ and the velocity diameter $d_V(\cdot)$ as follows $$d_{X}(t):=\max_{i,j=1,\dots,N}\lvert x_{i}(t)-x_{j}(t)\rvert,\quad \forall t\geq -\bar\tau$$
$$d_{V}(t):=\max_{i,j=1,\dots,N}\lvert v_{i}(t)-v_{j}(t)\rvert,\quad \forall t\geq -\bar\tau.$$
\begin{defn} \label{unflock} A solution $\{(x_{i},v_{i})\}_{i=1,\dots,N}$ to system \eqref{onoff} exhibits \textit{asymptotic flocking} if the two following conditions are satisfied:
	\begin{enumerate}
		\item there exists a positive constant $d^{*}$ such that$$\sup_{t\geq-{\bar\tau}}d_{X}(t)\leq d^{*};$$
		\item$\underset{t \to+\infty}{\lim}d_{V}(t)=0.$
	\end{enumerate}
\end{defn}
Our main result is the following.
\begin{thm} \label{uf}
	Assume \eqref{taubounded}. Let $\psi:\RR\rightarrow\RR$ be a positive, bounded, continuous function that satisfies
	\begin{equation}\label{infint}
		\int_{0}^{+\infty}\min_{r\in [0,x]}\psi(r)dx=+\infty.
	\end{equation}
	Assume that the weight function $\alpha:[0,+\infty)\rightarrow[0,1]$ is a ${\mathcal L^1}$-measurable satisfying the Persistence Excitation Condition {\bf (PE)}. Moreover, let $x^{0}_{i},v_{i}^{0}:[-{\bar\tau},0]\rightarrow \RR^{d}$ be continuous functions, for any $i=1,\dots,N$. 
	Then, for every solution $\{(x_{i},v_{i})\}_{i=1,\dots,N}$ to \eqref{onoff} with the initial conditions \eqref{incond}, there exists a positive constant $d^{*}$ such that \begin{equation}\label{posbound}
		\sup_{t\geq-{\bar\tau}}d_{X}(t)\leq d^{*},
	\end{equation}
	and there exists another positive constant $\mu$, independent of $N$, for which the following exponential decay estimate holds
	\begin{equation}\label{vel}
		d_{V}(t)\leq \left(\max_{i,j=1,\dots,N}\,\,\max_{r,s\in [-\bar\tau,0]}\lvert v_{i}(r)-v_{j}(s)\rvert\right) e^{-\mu(t-3T)},\quad \forall t\geq 0,
	\end{equation}
	where $T>0$ is the constant in the Persistence Excitation Condition {\bf (PE)}.
\end{thm}
\begin{oss}
	Let us note that, if the function ${\psi}$ is nonincreasing, then the condition \eqref{infint} reduces to
	$$\int_{0}^{+\infty}\psi(x)dx=+\infty,$$
	which is the classical assumption used to obtain the unconditional flocking (see e.g. \cite{Cartabia}). Since here we deal with an influence function not necessarily monotonic, we require the stronger assumption \eqref{infint} (cf. \cite{Cont}).
\end{oss}
	\section{The Cucker-Smale model with time delay and communication failures}\label{main}
	In this section, we prove Theorem \ref{uf}. Firstly, we will present some preliminary notions and lemmas. In particular, suitable estimates on the agents' velocities are deduced. Then, we will move to the proof of Theorem \ref{uf}, based on a Lyapunov functional approach.
	\subsection{Preliminary lemmas}
		Let $\{x_{i},v_i\}_{i=1,\dots,N}$ be solution to \eqref{onoff} with the initial conditions \eqref{incond}. We present some preliminary lemmas. We omit their proofs, since they can be proved using the same arguments employed in \cite{ContPi}. The only difference with the analogous in \cite{ContPi} is that, here, the estimates provided in \cite{ContPi} for the agents' positions are rather valid for the agents' velocities and for the velocity diameters. 
	The first result provides suitable estimates from below and above on the scalar products between agents' velocities and vectors in $\RR^{d}$.
	\begin{lem}\label{L1}
		For each $v\in \RR^{d}$ and $S\ge 0,$  we have that 
		\begin{equation}\label{scalpr}
			\min_{j=1,\dots,N}\min_{s\in[S-\bar\tau,S]}\langle v_{j}(s),v\rangle\leq \langle v_{i}(t),v\rangle\leq \max_{j=1,\dots,N}\max_{s\in[S-\bar\tau,S]}\langle v_{j}(s),v\rangle,
		\end{equation}for all  $t\geq S-\bar\tau$ and $i=1,\dots,N$. 
	\end{lem}
	Now, we introduce the following notations.
	\begin{defn}\label{successioneDn}
		We define, for all $n\in \mathbb{N}_{0}$,
		
		$$D_{n}:=\max_{i,j=1,\dots,N}\,\,\max_{s,t\in [nT-\bar\tau,nT ]}\lvert v_{i}(s)-v_{j}(t)\rvert.$$
	\end{defn}
	In particular, for $n=0$,  
	
	$$D_{0}=\max_{i,j=1,\dots,N}\,\,\max_{s, t\in [-\bar\tau, 0]}\lvert v_{i}(s)-v_{j}(t)\rvert.$$
	Note that $D_0$ coincides with the positive constant, depending of the initial data, appearing in the right-hand side of the exponential decay estimate \eqref{vel}. Thus, \eqref{vel} now reads as
	\begin{equation}\label{decayrewritten}
		d_V(t)\leq D_0 e^{-\mu(t-3T)},\quad \forall t\geq 0.
	\end{equation}
	
	The following lemmas let us deduce some properties of the sequence of generalized velocity diameters $\{D_{n}\}_{n}$.
	\begin{lem}
		For each $n\in \mathbb{N}_{0}$ and $i,j=1,\dots,N$, we get \begin{equation}\label{distgen}
			\lvert v_{i}(s)-v_{j}(t)\rvert\leq D_{n}, \quad\forall s,t\geq nT-\bar\tau.
		\end{equation} 
	\end{lem}
	\begin{oss}
		Let us note that estimate \eqref{distgen} implies that
		\begin{equation}\label{dx}
			\lvert v_{i}(s)-v_{j}(t)\rvert\leq D_{0}, \quad\forall s,t\geq -\bar\tau,
		\end{equation}
		and that
		\begin{equation}\label{d_v}
			d_V(t)\leq D_n,\quad \forall t\geq nT-\bar\tau, \,\,\forall n\in \mathbb{N}_0.
		\end{equation}
		Also, from \eqref{distgen} the sequence $\{D_n\}_n$ is nonincreasing, namely
		\begin{equation}\label{decgen}
			D_{n+1}\leq D_{n},\quad \forall n\in \mathbb{N}_{0}.
		\end{equation}
	\end{oss} 
	Now, we find a bound on $\vert v_i(t)\vert,$ which is uniform with respect to $t$ and $i=1,\dots,N.$ 
	\begin{lem}\label{L3}
		For every $i=1,\dots,N,$ we have that \begin{equation}\label{boundsol}
			\lvert v_{i}(t)\rvert\leq C_0^{V},\quad \forall t\geq-\bar\tau,
		\end{equation}
		where $C_0^{V}$ is the positive constant given by 
		\begin{equation}\label{C0}
			C_0^V:=\max_{j=1,\dots,N}\max_{s\in [-\bar\tau,0]}\lvert v_{j}(s)\rvert.
		\end{equation}
	\end{lem}
	
	Let us note that the above result does not allow us to deduce a bound from below for the communication rates, as it happens for the first-order model (see \cite{ContPi}). Indeed, in the second-order model the communication rates depend on the agents' positions, whereas estimate \eqref{boundsol} provides the uniform boundedness of the agents' velocities. To get the existence of a lower bound on the influence function, we have to prove that the position diameters are bounded. This will be proven later on by introducing a suitable Lyapunov functional.
	\medskip
	
	Now, the following lemma, that will be used to prove suitable estimates for the sequence of diameters $\{d_V(nT)\}_n$, holds.
	\begin{lem}
		For all $i,j=1,\dots,N$,  unit vector $v\in \RR^{d}$ and $n\in\mathbb{N}_{0}$, we have
		\begin{equation}\label{4gen}
			\langle v_{i}(t)-v_{j}(t),v\rangle\leq e^{-K(t-\bar{t})}\langle v_{i}(\bar{t})-v_{j}(\bar{t}),v\rangle+(1-e^{-K(t-\bar{t})})D_{n},
		\end{equation}
		for all $t\geq \bar{t}\geq nT$. \\Moreover, for all $n\in \mathbb{N}_{0}$, we get
		\begin{equation}\label{n+1gen}
			D_{n+1}\leq e^{-KT}d_V(nT)+(1-e^{-KT})D_{n}.
		\end{equation}
	\end{lem}
	To conclude, we state the following result, that will allow us to derive a bound from below on the communication rates.
	\begin{lem}
		For every $i,j=1,\dots,N$, we get
		\begin{equation}\label{dist}
			\lvert x_{i}(t)-x_{j}(t-\tau(t))\rvert\leq \bar\tau C_{0}^{V}+M^{X}_{0}+d_{X}(t), \quad\forall t\geq0,
		\end{equation}
		where $C_{0}^{V}$ is the positive constant in \eqref{C0}and $M^{X}_{0}$ is the positive constant given by
		\begin{equation}\label{MX}
			M_0^{X}:=\max_{l=1,\dots,N}\,\,\max_{s,t\in [-\bar\tau,0]}\lvert x_{l}(s)-x_l(r)\rvert.
		\end{equation}
	\end{lem}
	\subsection{Flocking estimate}
	Finally, we need the following crucial result. We obtain suitable estimates on the sequence $\{D_n\}_n$ that are independent of the number of agents, differently from \cite{CicoContPi}. The proof of this result combines arguments from \cite{CicoContPi,ContPi,Cont}. Before stating it, we give the following definition.
\begin{defn}
	We define
	$$\phi(t):=\min\left\{\psi(r):r\in \left[0,\bar\tau C^{V}_{0}+M^{X}_{0}+\max_{s\in[-\bar\tau,t] }d_{X}(s)\right]\right\},$$
	for all $t\geq -\bar\tau$.
\end{defn}
\begin{oss}
	Let us note that from \eqref{dist} it comes that
	$$\psi(\lvert x_i(t)-x_j(t-\tau(t))\rvert)\geq \phi(t),\quad\forall t\geq 0,\,\forall i,j=1,\dots,N,$$
	from which
	\begin{equation}\label{weightlowerbound}
		b_{ij}(t)\geq \frac{1}{N-1}\phi(t),\quad \forall t\geq 0,\,\forall i,j=1,\dots,N.
	\end{equation}
\end{oss}
\begin{prop}\label{lemma3}
	For each integer $n\geq 2$, we have that
	\begin{equation}\label{n-2}
		D_{n+1}\leq (1-C_n)D_{n-2},
	\end{equation}
where $C_n\in (0,1)$ is the constant, independent of $N$, given by
\begin{equation}\label{Cind}
	C_n:=e^{-KT}\min\left\{e^{-K(T+\bar\tau)},e^{-KT}\phi(nT)\tilde{\alpha}\right\}.
\end{equation}
\end{prop}
\begin{proof}
	First of all, we claim that
	\begin{equation}\label{dv}
		d_{V}(nT)\leq (1-C_n^*)D_{n-2},\quad\forall n\geq 2,
	\end{equation}
where $$C_n^*:=\min\left\{e^{-K(T+\bar\tau)},e^{-KT}\phi(nT)\tilde{\alpha}\right\}.$$
	To this aim, let $n\geq2$. Note that, if $d_{V}(nT)=0$, \eqref{dv} is trivially satisfied.
	So, we can assume $d_{V}(nT)>0$. Let $i,j=1,\dots,N$ be such that $$d_{V}(nT)=\lvert v_{i}(nT)-v_{j}(nT)\rvert.$$
	We set $$v=\frac{v_{i}(nT)-v_{j}(nT)}{\lvert v_{i}(nT)-v_{j}(nT)\rvert}.$$
	Then $v$ is a unit vector for which we can write $$d_{V}(nT)=\langle v_{i}(nT)-v_{j}(nT),v\rangle.$$
	At this point, we distinguish two cases.
	\\\par\textit{Case I.} Assume that there exists $t_{0}\in [nT-T-
	\bar\tau,nT]$ such that $$\langle v_{i}(t_0)-v_{j}(t_0),v\rangle<0.$$
	Note that $nT-T-{\bar\tau}\geq nT-2T$ due to $T\geq \bar{\tau}$ (see {\bf (PE)}) and that $nT\geq t_0\geq nT-T-{\bar\tau}\geq nT-2T$. Then, by using \eqref{4gen} with $t=nT$ and $\bar{t}=t_0$, we have  
	\begin{equation}\label{firstestimate}
		\begin{split}
			d_V(nT)&\leq e^{-K(nT-t_{0})}\langle v_{i}(t_0)-v_{j}(t_0),v\rangle+(1-e^{-K(nT-t_{0})})D_{n-2}\\&<(1-e^{-K(nT-t_{0})})D_{n-2}\leq (1-e^{-K(T+{\bar\tau})})D_{n-2}.
		\end{split}
	\end{equation}
	\par \textit{Case II.}  Assume that \begin{equation}\label{scalpos}
		\langle v_{i}(t)-v_{j}(t),v\rangle\geq0,\quad \forall t\in [nT-T-\bar{\tau},nT]
	\end{equation}
We set
	$$M_{n-1}:=\max_{l=1,\dots,N}\max_{s\in [nT-T-{\bar\tau},nT-T]}\langle v_{l}(s),v\rangle,$$
	$$m_{n-1}:=\min_{l=1,\dots,N}\min_{s\in [nT-T-{\bar\tau},nT-T]}\langle v_{l}(s),v\rangle.$$
	Then, $M_{n-1}-m_{n-1}\leq D_{n-1}$. Also, for a.e. $t\in [nT-T,nT]$, it holds
	$$\begin{array}{l}
		\vspace{0.3cm}\displaystyle{\frac{d}{dt}\langle v_{i}(t)-v_{j}(t),v\rangle=\sum_{l:l\neq i}\alpha(t)b_{il}(t)\langle v_{l}(t-\tau(t))-v_{i}(t),v\rangle}\\
		\vspace{0.3cm}\displaystyle{\hspace{4cm}+\sum_{l:l\neq j}\alpha(t)b_{jl}(t)\langle v_{j}(t)-v_{l}(t-\tau(t)),v\rangle}\\
		\vspace{0.3cm}\displaystyle{\hspace{3.2cm}=\sum_{l:l\neq i}\alpha(t)b_{il}(t)(\langle v_{l}(t-\tau(t)),v\rangle-M_{n-1}+M_{n-1}-\langle v_{i}(t),v\rangle)}\\
		\vspace{0.3cm}\displaystyle{\hspace{4cm}+\sum_{l:l\neq j}\alpha(t)b_{jl}(t)(\langle v_{j}(t),v\rangle-m_{n-1}+m_{n-1}-\langle v_{l}(t-\tau(t)),v\rangle)}\\
		\displaystyle{\hspace{3 cm}:=S_{1}+S_{2}.}
	\end{array}$$
	We recall that, from \eqref{scalpr}, $$m_{n-1}\leq \langle v_{k}(s),v\rangle\leq M_{n-1},\quad \forall s\geq nT-T-{\bar\tau},\,\forall k=1,\dots,N.$$ Combining this last fact with \eqref{weightlowerbound}, for a.e. $t\in [nT-T,nT]$, since $\alpha(t)\leq 1$, it comes that
	$$\begin{array}{l}
		\vspace{0.3cm}\displaystyle{S_{1}= \sum_{l:l\neq i}\alpha(t)b_{il}(t)(\langle v_{l}(t-\tau(t)),v\rangle-M_{n-1})+\sum_{l:l\neq i}\alpha(t)b_{il}(t)(M_{n-1}-\langle v_{i}(t),v\rangle)}\\
		\vspace{0.3cm}\displaystyle{\hspace{0.6cm}\leq \frac{\phi(t)\alpha(t)}{N-1}\sum_{l:l\neq i}(\langle v_{l}(t-\tau(t)),v\rangle-M_{n-1})+\frac{K}{N-1}\sum_{l:l\neq i}(M_{n-1}-\langle v_{i}(t),v\rangle)}\\
		\displaystyle{\hspace{0.6cm}=\frac{\phi(t)\alpha(t)}{N-1}\sum_{l:l\neq i}(\langle v_{l}(t-\tau(t)),v\rangle-M_{n-1})+K(M_{n-1}-\langle v_{i}(t),v\rangle),}
	\end{array}$$
	and
	$$\begin{array}{l}
		\vspace{0.3cm}\displaystyle{S_{2}=\sum_{l:l\neq j}\alpha(t)b_{jl}(t)(\langle v_{j}(t),v\rangle-m_{n-1})+\sum_{l:l\neq j}\alpha(t)b_{jl}(t)(m_{n-1}-\langle v_{l}(t-\tau(t)),v\rangle)}\\
		\vspace{0.3cm}\displaystyle{\hspace{0.6cm}\leq \frac{K}{N-1}\sum_{l:l\neq j}(\langle v_{j}(t),v\rangle-m_{n-1})+\frac{\phi(t)\alpha(t)}{N-1}\sum_{l:l\neq j}(m_{n-1}-\langle v_{l}(t-\tau(t)),v\rangle)}\\
		\displaystyle{\hspace{0.6cm}=K(\langle v_{j}(t),v\rangle-m_{n-1})+\frac{\phi(t)\alpha(t)}{N-1}\sum_{l:l\neq j}(m_{n-1}-\langle v_{l}(t-\tau(t)),v\rangle).}
	\end{array}$$
	Hence, we get
	$$\begin{array}{l}
		\vspace{0.3cm}\displaystyle{S_{1}+S_{2}\leq K(M_{n-1}-\langle v_{i}(t),v\rangle+\langle v_{j}(t),v\rangle-m_{n-1})}\\
		\vspace{0.3cm}\displaystyle{\hspace{2.5cm}+\frac{\phi(t)\alpha(t)}{N-1}\sum_{l:l\neq i,j}(\langle v_{l}(t-\tau(t)),v\rangle-M_{n-1}+m_{n-1}-\langle v_{l}(t-\tau(t)),v\rangle)}\\
		\vspace{0.3cm}\displaystyle{\hspace{2.5cm}+\frac{\phi(t)\alpha(t)}{N-1}(\langle v_{j}(t-\tau(t)),v\rangle-M_{n-1}+m_{n-1}-\langle v_{i}(t-\tau(t)),v\rangle)}\\
		\vspace{0.3cm}\displaystyle{\hspace{1.4cm}=K(M_{n-1}-m_{n-1}-\langle v_i(t)-v_j(t),v\rangle)+\frac{\phi(t)\alpha(t)}{N-1}(N-2)(m_{n-1}-M_{n-1})}\\
		\displaystyle{\hspace{2.5cm}+\frac{\phi(t)\alpha(t)}{N-1}(m_{n-1}-M_{n-1}-\langle v_i(t-\tau(t))-v_j(t-\tau(t)),v\rangle).}
	\end{array}$$
	Now, by virtue of \eqref{scalpos}, $\langle v_i(t-\tau(t))-v_j(t-\tau(t)),v\rangle\geq 0$.  Thus, recalling that $M_{n-1}-m_{n-1}\leq D_{n-1}$, we get  
	$$\begin{array}{l}
		\vspace{0.3cm}\displaystyle{\frac{d}{dt}\langle v_i(t)-v_j(t),v\rangle\leq K(M_{n-1}-m_{n-1}-\langle v_i(t)-v_j(t),v\rangle)}\\
		\vspace{0.3cm}\displaystyle{\hspace{4cm}+\phi(t)\alpha(t)\left(\frac{N-2}{N-1}+\frac{1}{N-1}\right)(m_{n-1}-M_{n-1})}\\
		\vspace{0.3cm}\displaystyle{\hspace{3cm}=(K-\phi(t)\alpha(t))(M_{n-1}-m_{n-1})-K\langle v_i(t)-v_j(t),v\rangle}\\
		\displaystyle{\hspace{3cm}\leq (K-\phi(t)\alpha(t))D_{n-1}-K\langle v_i(t)-v_j(t),v\rangle,}
	\end{array}$$
	for a.e $t\in [nT-T,nT]$.
	Then, Gronwall's inequality yields 
	$$\begin{array}{l}
		\vspace{0.3cm}\displaystyle{\langle v_i(t)-v_j(t),v\rangle\leq 
		e^{-K(t-nT+T)}\langle v_i(nT-T)-v_j(nT-T),v\rangle}\\
	\vspace{0.3cm}\displaystyle{\hspace{3.5cm}+D_{n-1}\int_{nT-T}^{t}(K-\alpha(s)\phi(s))e^{-K(t-s)}ds}\\
		\vspace{0.3cm}\displaystyle{\hspace{2.7cm}= e^{-K(t-nT+T)}\langle v_i(nT-T)-v_j(nT-T),v\rangle}\\
		\displaystyle{\hspace{3.5cm}+D_{n-1}\left(1-e^{-K(t-nT+T)}-\int_{nT-T}^{t}e^{-K(t-s)}\alpha(s)\phi(s)ds\right),}
	\end{array}$$
	for every $t\in [nT-T,nT]$. In particular, for $t=nT$, we find
	 $$\begin{array}{l}
		\vspace{0.3cm}\displaystyle{d_{V}(nT)\leq e^{-KT}\langle v_i(nT-T)-v_j(nT-T),v\rangle}\\
		\vspace{0.3cm}\displaystyle{\hspace{2.5cm}+D_{n-1}\left(1-e^{-KT}-\int_{nT-T}^{nT}e^{-K(nT-s)}\alpha(s)\phi(s)ds\right)}\\
		\displaystyle{\hspace{1.7cm}\leq \left(1-e^{-KT}\int_{nT-T}^{nT}\alpha(s)\phi(s)ds\right)D_{n-1},}
	\end{array}$$
	where here we have used the fact that $\langle v_i(nT-T)-v_j(nT-T),v\rangle\leq D_{n-1}$. Notice that $\phi(\cdot)$ is a nonincreasing function, so that $$\phi(s)\geq \phi(nT),\quad\forall s\in [nT-T,nT].$$ 
	Thus, the above inequality becomes
	$$d_{V}(nT)\leq  \left(1-e^{-KT}\phi(nT)\int_{nT-T}^{nT}\alpha(s)ds\right)D_{n-1},$$
	from which, using the Persistence Excitation condition {\bf (PE)}, we can write
	\begin{equation}\label{secondestimate}
		d_{V}(nT)\leq  \left(1-e^{KT}\phi(nT)\tilde{\alpha}\right)D_{n-1}.
	\end{equation}
	So, taking into account the two estimates \eqref{firstestimate} and \eqref{secondestimate}, we deduce that \eqref{dv} is fulfilled.
	\\Now, we prove \eqref{n-2}. Let $n\geq 2$. Then, from \eqref{decgen}, \eqref{n+1gen} and \eqref{dv}, it immediately follows that
	$$\begin{array}{l}
		\vspace{0.3cm}\displaystyle{D_{n+1}\leq e^{-KT}d_{V}(nT)+(1-e^{-KT})D_{n}\leq e^{-KT}(1-C_n^*)D_{n-2}+(1-e^{-KT})D_{n}}\\
\displaystyle{\hspace{1.5cm}\leq (e^{-KT}-e^{-KT}C_n^*+1-e^{-KT})D_{n-2}=(1-e^{-KT}C_n^*)D_{n-2},}
\end{array}$$
	i.e. \eqref{n-2} holds true.
\end{proof}

Now, we are ready to prove Theorem \ref{uf}.
\begin{proof}[Proof of Theorem \ref{uf}]
	Let $\{(x_{i},v_{i})\}_{i=1,\dots,N}$  be solution to \eqref{onoff} under the initial conditions \eqref{incond}. Following \cite{CicoContPi}, let us define
	$$\tilde{C}_{n}=\frac{C_n}{T}, \quad\forall n\in \mathbb{N}_0,$$
	where $C_n\in (0,1)$ is defined in \eqref{Cind}.
	\\Let us introduce the function $\mathcal{E}:[-\bar\tau,+\infty)\rightarrow[0,+\infty) ,$ 
	$$\mathcal{E}(t):=\begin{cases}
		D_{0}, \quad&t\in [-\bar\tau,2T),\\D_{3n}\left(1-\tilde C_{3n+2}(t-nT)\right), \quad &t\in [nT,(n+1)T),\,n\geq 2.
	\end{cases}$$
	By definition, $\mathcal{E}$ is piecewise continuous and positive. Also, $\mathcal{E}$ is nonincreasing in every interval $[nT,(n+1)T)$, $n\geq 2,$ namely 
	$$\mathcal{E}(t)\leq \mathcal{E}(s),\quad \forall \,nT\leq s\leq t< (n+1)T,\quad \forall n\geq 2.$$
	So, 
	$$D_{3n}(1-C_{3n+2})=D_{3n}(1-\tilde{C}_{3n+2}T)=\lim_{t\to (n+1)T^-}\mathcal{E}(t)\leq\mathcal{E}(nT).$$
	On the other hand, $$\mathcal{E}((n+1)T)=D_{3(n+1)}=D_{3n+3}=D_{(3n+2)+1},$$
	from which \eqref{n-2} yields
	$$\mathcal{E}((n+1)T)\leq (1-C_{3n+2})D_{3n}.$$
	In particular,
	\begin{equation}\label{En+1}
		\mathcal{E}((n+1)T)\leq \lim_{t\to (n+1)T^-}\mathcal{E}(t)\leq \mathcal{E}(nT).
	\end{equation}
	\\Now, for almost all time (see \cite{Cont} for further details)
	\begin{equation}\label{derdiamX}
		\frac{d}{dt}\max_{s\in [-{\bar\tau},t]}d_{X}(s)\leq \left\lvert \frac{d}{dt}d_{X}(t)\right\rvert\leq  d_{V}(t).
	\end{equation}
	We define the function $\mathcal{W}:[-\bar\tau,+\infty)\rightarrow [0,+\infty)$, 
	$$
	\mathcal{W}(t):=T\mathcal{E}(t)+\frac{e^{-KT}}{3}\int_{0}^{\bar\tau C^{V}_{0}+M^{X}_{0}+\underset{s\in [-\bar\tau,3t+2T]}{\max}d_{X}(s)}\min\left\{e^{-K(T+\bar\tau)},e^{-KT}\tilde{\alpha}\min_{\sigma\in [0,r]}\psi(\sigma)\right\}\,dr,
	$$
	for all $t  \geq -\bar\tau$. By construction, $\mathcal{W}$ is piecewise continuous. Also, for each $n\geq 2$ and for a.e. $t\in(nT,(n+1)T) $, from \eqref{d_v} and \eqref{derdiamX} it follows that 
	$$\begin{array}{l}
		\vspace{0.3cm}\displaystyle{\frac{d}{dt}\mathcal{W}(t)=T\frac{d}{dt}\mathcal{E}(t)+\frac{1}{3}e^{-KT}\min\left\{e^{-K(T+\bar\tau)},e^{-KT}\phi(3t+2T)\tilde{\alpha}\right\}\frac{d}{dt}\underset{s\in [-\bar\tau,3t+2T]}{\max}d_{X}(s)}\\
		\vspace{0.3cm}\displaystyle{\hspace{1.5cm}\leq-TD_{3n}\tilde{C}_{3n+2}+\frac{1}{3}3e^{-KT}\min\left\{e^{-K(T+\bar\tau)},e^{-KT}\phi(3t+2T)\tilde{\alpha}\right\}d_V(3t+2T)}\\
		\displaystyle{\hspace{1.5cm}\leq D_{3n}\left(e^{-KT}\min\left\{e^{-K(T+\bar\tau)},e^{-KT}\phi(3nT+2T)\tilde{\alpha}\right\}- C_{3n+2}\right)= 0,}
	\end{array}$$
	where here we have used the fact that, being $\phi$ a nonincreasing function, $\phi(3t+2T)\leq \phi(3nT+2T)$.
	Then, \begin{equation}\label{negder}
		\frac{d}{dt}\mathcal{W}(t)\leq0,\quad \text{a.e. }t\in (nT,(n+1)T),
	\end{equation}
	which implies \begin{equation}\label{2tau1}
		\mathcal{W}(t)\leq \mathcal{W}(nT),\quad \forall t\in [nT,(n+1)T).
	\end{equation}
Letting $t\to (n+1)T^-$ in the above inequality, due to \eqref{En+1}, we get
$$\begin{array}{l}
	\vspace{0.3cm}\displaystyle{\mathcal{W}(nT)\geq \lim_{t\to (n+1)T^- }\mathcal{W}(t)=T\lim_{t\to (n+1)T^-}\mathcal{E}(t)}\\
	\vspace{0.3cm}\displaystyle{\hspace{2cm}+\frac{1}{3}\int_{0}^{\bar\tau C^{V}_{0}+M^{X}_{0}+\underset{s\in [-\bar\tau,3(n+1)T+2T]}{\max}d_{X}(s)}\min\left\{e^{-K(T+\bar\tau)},e^{-KT}\tilde{\alpha}\min_{\sigma\in [0,r]}\psi(\sigma)\right\}\,dr}\\
	\vspace{0.3cm}\displaystyle{\hspace{0.5cm}\geq T\mathcal{E}((n+1)T)+\frac{1}{3}\int_{0}^{\bar\tau C^{V}_{0}+M^{X}_{0}+\underset{s\in [-\bar\tau,3(n+1)T+2T]}{\max}d_{X}(s)}\min\left\{e^{-K(T+\bar\tau)},e^{-KT}\tilde{\alpha}\min_{\sigma\in [0,r]}\psi(\sigma)\right\}\,dr}\\
	\displaystyle{\hspace{0.5cm}=\mathcal{W}((n+1)T).}
\end{array}$$
We have so proven that $\mathcal{W}$ is nonincreasing in all intervals of the form $[nT,(n+1)T)$, $n\geq 2$, and that
\begin{equation}
	\mathcal{W}((n+1)T)\leq \mathcal{W}(nT),\quad \forall n\geq 2.
\end{equation}
Thus, $$\mathcal{W}(t)\leq \mathcal{W}(2T),\quad \forall t\geq 2T,$$
from which, being $\mathcal{E}$ a nonnegative function,
	$$\frac{1}{3}\int_{0}^{\bar\tau C^{V}_{0}+M^{X}_{0}+\underset{s\in [-\bar\tau,3t+2T]}{\max}d_{X}(s)}\min\left\{e^{-K(T+\bar\tau)},e^{-KT}\tilde{\alpha}\min_{\sigma\in [0,r]}\psi(\sigma)\right\}\,dr\leq\mathcal{W}(2T),
	$$
	for all $t\geq2T$. Letting $t\to \infty$ in the above inequality, we can write  \begin{equation}\label{lim2}
		\frac{1}{3}\int_{0}^{\bar\tau C^{V}_{0}+M^{X}_{0}+\underset{s\in [-\bar\tau,+\infty)}{\sup}d_{X}(s)}\min\left\{e^{-K(T+\bar\tau)},e^{-KT}\tilde{\alpha}\min_{\sigma\in [0,r]}\psi(\sigma)\right\}\,dr\leq\mathcal{W}(2T).
	\end{equation} 
	Finally, since the function $\psi$ satisfies property \eqref{infint}, arguing as in \cite{Cont}, from \eqref{lim2} we can conclude that there exists a positive constant $d^{*}$ such that \begin{equation}\label{firstcond}
		\bar\tau C^{V}_{0}+M^{X}_{0}+\underset{s\in [-\bar\tau,+\infty)}{\sup}d_{X}(s)\leq d^{*}.
	\end{equation}
	Now, let us define
	$$\hat\phi:=\min_{r\in[0,d^{*}]}\psi(r).$$
	Note that $\hat{\phi}>0$. Also, \eqref{firstcond} yields 
	\begin{equation}\label{lowerboundphi}
		\hat\phi\leq \phi(t), \quad \forall t\geq -\bar\tau.
	\end{equation}
	Then, from \eqref{n-2} 
	and  \eqref{lowerboundphi} we have 
	\begin{equation}\label{decunif}
		D_{n+1}\leq (1-\hat C )D_{n-2},\quad\forall n\geq 2,
	\end{equation}
where $$\hat{C}:=e^{-KT}\min\left\{e^{-K(T+\bar\tau)},e^{-KT}\hat{\phi}\tilde{\alpha}\right\}.$$
	Thus, thanks to an induction argument, we can write
	$$D_{3n}\leq (1-\hat{C} )^nD_0,\quad \forall n\in \mathbb{N}_0.$$
	Note that the above inequality can be rewritten as
	\begin{equation}\label{finalmente}
		D_{3n}\leq e^{-3nT\mu}D_0,\quad \forall n\in\mathbb{N}_0,
	\end{equation}
	where $$\mu=\frac{1}{3T}\ln\left(\frac{1}{1- \hat{C}}\right).$$
	Finally, let $t\geq 0$. Then, $t\in [3nT,3(n+1)T]$, for some $n\in \mathbb{N}_0$. Then, using \eqref{distgen} and \eqref{finalmente}
	$$d_{V}(t)\leq D_{3n}\leq e^{-3n\mu T}D_0 \leq e^{-\mu(t-3T)}D_0.$$
	This proves \eqref{decayrewritten}, i.e. the proof is concluded.
\end{proof}
\section{A model with distributed time delay}\label{distributed}
In the previous sections, we have focused on a Cucker-Smale type model with pointwise time delay. Now, we deal with the following Cucker-Smale type model with distributed time delay: \begin{equation}\label{hkd}
	\begin{cases}
		\frac{d}{dt}x_i(t)=v_{i}(t),\quad &t>0,\quad \forall i=1,\dots,N,
		\\\frac{d}{dt}v_{i}(t)=\frac{1}{h(t)}\underset{j:j\neq i}{\sum}\alpha(t)\int_{t-\tau_{2}(t)}^{t-\tau_{1}(t)}\beta(t-s)c_{ij}(t;s) (v_{j}(s)-v_{i}(t)) ds,\quad &t>0,\quad \forall i=1,\dots,N.
	\end{cases}
\end{equation}
Here, the time delay functions  $\tau_{1}:[0,+\infty)\rightarrow[0,+\infty)$, $\tau_{2}:[0,+\infty)\rightarrow[0,+\infty)$ are continuous functions and satisfy \begin{equation}\label{delay_distr}
	0\leq\tau_{1}(t)<\tau_{2}(t)\leq\bar\tau, \quad\forall t\geq 0,
\end{equation} 
for some positive constant $\bar\tau$.
\\The communication rates $c_{ij}(t;s)$ are of the form
\begin{equation}\label{weightd}
	c_{ij}(t;s):=\frac{1}{N-1}\psi(\lvert x_{i}(t)- x_{j}(s)\rvert), \quad\forall t\geq 0,\, \quad\forall i,j=1,\dots,N,
\end{equation}
where, as before, $\psi:\RR\rightarrow \RR$ is a positive, continuous, and bounded influence function, with supremum norm \eqref{K}. 
\\Moreover, $\beta:[0,\bar{\tau}]\rightarrow (0,+\infty)$ is a continuous function and
\begin{equation}\label{h(t)}
	h(t):=\int_{\tau_{1}(t)}^{\tau_{2}(t)}\beta(s)ds, \quad\forall t\geq 0.
\end{equation}
Note that, since we are assuming $\tau_1(t)<\tau_2(t)$ and $\beta (t)>0,$ $\forall t\geq 0,$ then the function $h(t)$ is always positive.
\\Again, the weight function $\alpha:[0,+\infty)\rightarrow [0,1]$ is a $\mathcal{L}^1$-measurable function satisfying the Persistence Excitation Condition {\bf (PE)}.
\\Moreover, due to the presence of the time delay, the initial conditions are continuous functions defined in the time interval $[-\bar{\tau},0]$ 
\begin{equation}\label{inconddist}
	x_{i}(s)=x^{0}_{i}(s),\quad v_{i}(s)=v^{0}_{i}(s), \quad  \forall s\in [-\bar\tau,0],\,\forall i=1,\dots,N.
\end{equation}

\medskip

In this section, we will establish an exponential flocking result for solutions to \eqref{hkd}. To this aim, following a similar approach to the one employed to show the exponential flocking for system \eqref{onoff}, we will first prove some preliminary results, deriving suitable estimates on the velocity diameters. 

Let $\{x_{i},v_i\}_{i=1,\dots,N}$ be solution to \eqref{onoff} with the initial conditions \eqref{inconddist}. Firstly, we prove the following result, analogous to Lemma \ref{L1}. We give some details of its proof since now we are dealing with the distributed time delay case and the approach in \cite{ContPi}, where pointwise time delays are considered, cannot be repeated verbatim, as rather happened for Lemma \ref{L1}.
	\begin{lem}
	For each $v\in \RR^{d}$ and $S\geq 0$, we have that 
	\begin{equation}\label{scalprdist}
		\min_{j=1,\dots,N}\min_{s\in[S-\bar\tau,S]}\langle v_{j}(s),v\rangle\leq \langle v_{i}(t),v\rangle\leq \max_{j=1,\dots,N}\max_{s\in[S-\bar\tau,S]}\langle v_{j}(s),v\rangle,
	\end{equation}for all  $t\geq S-\bar\tau$ and $i=1,\dots,N$. 
\end{lem}
\begin{proof}
	Given $S\geq 0$ and $v\in \RR^{d}$, we first note that the inequalities in \eqref{scalprdist} are satisfied for every $t\in [S-\bar\tau,S]$.
	\\Now, let us fix $S\geq 0$ and $v\in \RR^{d}$. We set $$M_{S}:=\max_{j=1,\dots,N}\max_{s\in[S-\bar\tau,S]}\langle v_{j}(s),v\rangle.$$
	For all $\epsilon >0$, let us define
	$$K^{\epsilon}:=\{t>S :\max_{i=1,\dots,N}\langle v_{i}(s),v\rangle < M_{S}+\epsilon,\,\forall s\in [S,t)\}.$$
	By continuity, we have that $K^{\epsilon}\neq\emptyset$. Thus, denoted with $$S^{\epsilon}:=\sup K^{\epsilon},$$
	it holds that $S^{\epsilon}>S$. \\We claim that $S^{\epsilon}=+\infty$. Indeed, suppose by contradiction that $S^{\epsilon}<+\infty$. Note that by definition of $S^{\epsilon}$ it turns out that \begin{equation}\label{max}
		\max_{i=1,\dots,N}\langle v_{i}(t),v\rangle<M_{S}+\epsilon,\quad \forall t\in (S,S^{\epsilon}),
	\end{equation}
	and \begin{equation}\label{teps}
		\max_{i=1,\dots,N}\langle v_{i}(S^\epsilon),v\rangle=M_{S}+\epsilon.
	\end{equation}
	For all $i=1,\dots,N$ and for a.e. $t\in (S,S^{\epsilon})$, we compute
	$$\begin{array}{l}
		\vspace{0.3cm}\displaystyle{\frac{d}{dt}\langle v_{i}(t),v\rangle=\frac{1}{h(t)}\sum_{j:j\neq i}\alpha(t)\int_{t-\tau_{2}(t)}^{t-\tau_{1}(t)}\beta(t-s)c_{ij}(t;s)\langle v_{j}(s)-v_{i}(t),v\rangle ds}\\
		\displaystyle{\hspace{1cm}=\frac{1}{N-1}\frac{1}{h(t)}\sum_{j:j\neq i}\alpha(t)\int_{t-\tau_{2}(t)}^{t-\tau_{1}(t)}\beta(t-s)\psi(\lvert x_{i}(t)-x_{j}(s)\rvert)(\langle v_{j}(s),v\rangle-\langle v_{i}(t),v\rangle) ds.}
	\end{array}$$
	Notice that, for $t\in (S,S^{\epsilon})$, it holds $t-\tau_{2}(t),t-\tau_{1}(t)\in (S-\bar\tau, S^{\epsilon})$. Then, estimate \eqref{max} and the fact that the second inequality in \eqref{scalprdist} holds in the interval $[S-\bar{\tau},S]$ imply that
	\begin{equation}\label{t-tau}
		\langle v_{j}(s),v\rangle< M_{S}+\epsilon,\quad \forall s\in [t-\tau_{2}(t),t-\tau_{1}(t)],\,\forall t\in (S,S^\epsilon),\,\forall j=1, \dots, N.
	\end{equation}
	Therefore, from \eqref{h(t)}, \eqref{max} and \eqref{t-tau}, we can write 
	$$\begin{array}{l}
		\vspace{0.3cm}\displaystyle{\frac{d}{dt}\langle v_{i}(t),v\rangle\leq \frac{1}{N-1}\frac{1}{h(t)}\sum_{j:j\neq i}\alpha(t)\int_{t-\tau_{2}(t)}^{t-\tau_{1}(t)}\beta(t-s)\psi(\lvert x_{i}(t)-x_{j}(s)\rvert)(M_{S}+\epsilon-\langle v_{i}(t),v\rangle)ds}\\
		\vspace{0.3cm}\displaystyle{\hspace{2cm}\leq \frac{K}{N-1}\frac{1}{\int_{\tau_{1}(t)}^{\tau_{2}(t)}\beta(s)ds}\sum_{j:j\neq i}(M_{S}+\epsilon-\langle v_{i}(t),v\rangle)\int_{t-\tau_{2}(t)}^{t-\tau_{1}(t)}\beta(t-s)ds}\\
		\displaystyle{\hspace{2cm}=\frac{K}{\int_{\tau_{1}(t)}^{\tau_{2}(t)}\beta(s)ds}(M_{S}+\epsilon-\langle v_{i}(t),v\rangle)\int_{t-\tau_{2}(t)}^{t-\tau_{1}(t)}\beta(t-s)ds,}
	\end{array}$$
	for a.e. $t\in (S, S^{\epsilon}),$ where in the above inequalities we used that $\alpha\leq 1$. Now, with a simple change of variable,
	$$\int_{t-\tau_{2}(t)}^{t-\tau_{1}(t)}\beta(t-s)ds=\int_{\tau_{1}(t)}^{\tau_2(t)}\beta(s)ds.$$
	Thus, we obtain 
	 	$$\begin{array}{l}
	 	\vspace{0.3cm}\displaystyle{\frac{d}{dt}\langle v_{i}(t),v\rangle\leq K(M_{S}+\epsilon-\langle v_{i}(t),v\rangle),}
	 	\end{array}
	$$
	for a.e. $t\in (S, S^{\epsilon})$. Therefore, Gronwall's estimate yields
	$$\begin{array}{l}
		\vspace{0.3cm}\displaystyle{
			\langle v_{i}(t),v\rangle\leq e^{-K(t-S)}\langle v_{i}(S),v\rangle+(M_{S}+\epsilon)(1-e^{-K(t-S)})}\\
			\displaystyle{\hspace{1cm}\leq e^{-K(t-S)}M_S+(M_{S}+\epsilon)(1-e^{-K(t-S)})=M_S+\epsilon-\epsilon e^{-K(S^\epsilon-S)},}
	
	\end{array}
	$$for all $t\in (S, S^{\epsilon})$.	We have so proved that, $\forall i=1,\dots, N,$
	$$\langle v_{i}(t),v\rangle\leq M_{S}+\epsilon-\epsilon e^{-K(S^{\epsilon}-S)}, \quad \forall t\in (S,S^{\epsilon}).$$
	Thus, we get
	\begin{equation}\label{lim}
		\max_{i=1,\dots,N} \langle v_{i}(t),v\rangle\leq M_{S}+\epsilon-\epsilon e^{-K(S^{\epsilon}-S)}, \quad \forall t\in (S,S^{\epsilon}).
	\end{equation}
	Letting $t\to S^{\epsilon-}$ in \eqref{lim}, from \eqref{teps} we have that $$M_{S}+\epsilon\leq M_{S}+\epsilon-\epsilon e^{-K(S^{\epsilon}-S)}<M_{S}+\epsilon,$$
	which is a contradiction. Thus, $S^{\epsilon}=+\infty$ and from the arbitrariness of $\epsilon$ the second inequality in \eqref{scalprdist} is fulfilled. 
	\\The other inequality in \eqref{scalprdist} can be proven using the same arguments employed in \cite{Cont}.
\end{proof}
As in Section \ref{main}, we now introduce the following generalize diameters. 
\begin{defn}
	We define, for all $n\in \mathbb{N}_{0}$,
	$$F_{n}:=\max_{i,j=1,\dots,N}\,\,\max_{s,t\in [nT-\bar\tau,nT ]}\lvert v_{i}(s)-v_{j}(t)\rvert.$$
\end{defn}
In particular, for $n=0$,
$$F_{0}=\max_{i,j=1,\dots,N}\,\,\max_{s, t\in [-\bar\tau, 0]}\lvert v_{i}(s)-v_{j}(t)\rvert.$$ 
Once we have established the existence of invariant sets (see estimate \eqref{scalprdist}), the following properties, analogous to the ones in Section \ref{main}, can be deduced for the sequence of generalized diameters $\{F_n\}_n$. Since these properties can be deduced from \eqref{scalprdist} using the same approach in \cite{CicoContPi, ContPi, ContPign,Cont}, we omit their proofs.
\begin{lem}
	For each $n\in \mathbb{N}_{0}$ and $i,j=1,\dots,N$, we get \begin{equation}\label{distdist}
		\lvert v_{i}(s)-v_{j}(t)\rvert\leq F_{n}, \quad\forall s,t\geq nT-\bar\tau.
	\end{equation}
\end{lem}

\begin{oss}
	Let us note that from \eqref{distdist}, in particular, it follows that
	\begin{equation}
		\label{dxdist}
		\lvert v_{i}(s)-v_{j}(t)\rvert\leq F_{0}, \quad\forall s,t\geq -\bar{\tau}.
	\end{equation}
	Moreover, it holds
	\begin{equation}\label{decdist}
		F_{n+1}\leq F_{n},\quad \forall n\in \mathbb{N}_{0}.
	\end{equation}
\end{oss} 
Also, the agents' velocities are bounded, uniformly with respect to $t$ and $i=1,\dots,N$, by a positive constant that depends on the initial data.
\begin{lem}
	For every $i=1,\dots,N$ we have that \begin{equation}\label{boundsoldist}
		\lvert v_{i}(t)\rvert\leq R_{0}^V,\quad \forall t\geq-\bar{\tau},
	\end{equation}
	where 
	\begin{equation}\label{R0}
		R_0^V:=\max_{j=1,\dots,N}\,\,\max_{s\in [-\bar\tau, 0]}\lvert v_{j}(s)\rvert.
	\end{equation}
\end{lem}
\begin{lem}
	For all $i,j=1,\dots,N$,  unit vector $v\in \RR^{d}$ and $n\in\mathbb{N}_{0}$ we have that 
	\begin{equation}\label{4dist}
		\langle v_{i}(t)-v_{j}(t),v\rangle\leq e^{-K(t-\bar{t})}\langle v_{i}(\bar{t})-v_{j}(\bar{t}),v\rangle+(1-e^{-K(t-\bar{t})})F_{n},
	\end{equation}
	for all $t\geq \bar{t}\geq nT$. \\Moreover, for all $n\in \mathbb{N}_{0}$, we get
	\begin{equation}\label{n+1dist}
		F_{n+1}\leq e^{-KT}d_V(nT)+(1-e^{-KT})F_{n}.
	\end{equation}
\end{lem}
Now, analogously to Section \ref{main}, we need to find a bound from below on the communication rates. First of all, we prove the following result.  
\begin{lem}
	For every $i,j=1,\dots,N$, we get
	\begin{equation}\label{tassi}
		\lvert x_{i}(t)-x_{j}(s)\rvert\leq \bar\tau R_{0}^{V}+N^{X}_{0}+d_{X}(t), \quad\forall s\in [t-\tau_2(t),t-\tau_1(t)],\,\,\forall t\geq 0,
	\end{equation}
	where $R_{0}^{V}$ is the positive constant in \eqref{R0} and $N^{X}_{0}$ is the positive constant given by
	\begin{equation}\label{N0}
		N_0^{X}:=\max_{l=1,\dots,N}\,\,\max_{r,w\in [-\bar\tau,0]}\lvert x_{i}(r)-x_i(w)\rvert.
	\end{equation}
	\begin{proof}
		Given $i,j=1,\dots,N$ and $t\geq0$, for all $s\in [t-\tau_2(t),t-\tau_1(t)]$ we have
		\begin{equation}\label{splitdist}
			\begin{split}
				\lvert x_{i}(t)-x_{j}(s)\rvert&\leq \lvert x_{i}(t)-x_{j}(t)\rvert+\lvert x_{j}(t)-x_{j}(s)\rvert\\&\leq d_{X}(t)+\lvert x_{j}(t)-x_{j}(s)\rvert.
			\end{split}
		\end{equation}
		Now, we estimate $\lvert x_{j}(t)-x_{j}(s)\rvert.$ If $t-\tau_{2}(t)>0$, from \eqref{delay_distr} and \eqref{boundsoldist} we get
		$$\begin{array}{l}
			\displaystyle{\lvert x_{j}(t)-x_{j}(s)\rvert\leq \int_{s}^{t}\lvert v_{j}(r)\rvert dr\leq R^{V}_{0}(t-s)\leq R^{V}_{0}\tau_2(t)\leq \bar{\tau}R^{V}_{0}.}
		\end{array}$$
		On the other hand, if $t-\tau_{1}(t)\leq 0$, then $t\leq \bar\tau$ and 
		$$\begin{array}{l}
			\vspace{0.3cm}\displaystyle{\lvert x_{j}(t)-x_{j}(s)\rvert\leq\lvert x_j(0)-x_j(s)\rvert+ \int_{0}^{t}\lvert v_{j}(r)\rvert dr}\\
			\displaystyle{\hspace{2cm}\leq N^{X}_{0}+t R^{V}_{0}\leq N^{X}_0+\bar\tau R^{V}_{0}.}
		\end{array}$$
		Finally, if $t-\tau_2(t)\leq 0$ and $t-\tau_{1}(t)>0$, arguing as above we have
		$$\lvert x_{j}(t)-x_{j}(s)\rvert\leq \bar{\tau}R^{V}_{0},$$
		if $s>0$ and
		$$\lvert x_{j}(t)-x_{j}(s)\rvert\leq N^{X}_0+\bar\tau R^{V}_{0},$$
		if $s\leq 0$.
		Therefore, in all cases, for all $s\in [t-\tau_2(t),t-\tau_1(t)]$,
		$$\lvert x_{j}(t)-x_{j}(s)\rvert\leq N^{X}_0+\bar\tau R^{V}_{0},$$
		from which \eqref{splitdist} becomes
		$$	\lvert x_{i}(t)-x_{j}(s)\rvert\leq N^{X}_0+\bar\tau R^{V}_{0}+d_{X}(t).$$
	\end{proof}
\end{lem}
Next, we introduce the following auxiliary function.
\begin{defn}
	We define
	$$\eta(t):=\min\left\{\psi(r):r\in \left[0,\bar\tau R^{V}_{0}+N_{X}^{0}+\max_{s\in[-\bar\tau,t] }d_{X}(s)\right]\right\},$$
	for all $t\geq -\bar\tau$.
\end{defn}
\begin{oss}
	Let us note that from \eqref{tassi} it comes that
	$$\psi(\lvert x_i(t)-x_j(s)\rvert)\geq \eta(t),\quad\forall s\in [t-\tau_2(t),t-\tau_1(t)],\,\forall t\geq 0,\,\forall i,j=1,\dots,N,$$
	from which
	\begin{equation}\label{lowerbound}
		c_{ij}(t;s)\geq \frac{1}{N-1}\eta(t),\quad\forall s\in [t-\tau_2(t),t-\tau_1(t)],\,\forall t\geq 0,\,\forall i,j=1,\dots,N.
	\end{equation}
\end{oss}
Finally, analogously to Section \ref{main}, we prove the following crucial result, that will be used to prove the velocity alignment and the boundedness of the position diameters. The ideas behind the proof of this result are similar the ones used in Proposition \ref{lemma3}. However, we will show some details of its proof since the presence of the distributed time delay requires a more careful analysis.
\begin{prop}\label{lemma3dist}
	For each integer $n\geq 2$, we have that
	\begin{equation}\label{n-2dist}
		F_{n+1}\leq (1-\gamma_n)D_{n-2},
	\end{equation}
	where $\gamma_n\in (0,1)$ is the constant, independent of $N$, given by
	\begin{equation}\label{gammaind}
		\gamma_n:=e^{-KT}\min\left\{e^{-K(T+\bar\tau)},e^{-KT}\eta(nT)\tilde{\alpha}\right\}.
	\end{equation}
\end{prop}
\begin{proof}
Let us prove that
	\begin{equation}\label{dvdist}
		d_{V}(nT)\leq (1-\gamma_n^*)F_{n-2},\quad\forall n\geq 2,
	\end{equation}
	where $$\gamma_n^*:=\min\left\{e^{-K(T+\bar\tau)},e^{-KT}\eta(nT)\tilde{\alpha}\right\}.$$
	To see this, let $n\geq2$. Then, if $d_{V}(nT)=0$, obviously \eqref{dvdist} is trivially satisfied.
	So, we can assume $d_{V}(nT)>0$. Let $i,j=1,\dots,N$ be such that $$d_{V}(nT)=\lvert v_{i}(nT)-v_{j}(nT)\rvert.$$
	Let us consider the unit vector $$v=\frac{v_{i}(nT)-v_{j}(nT)}{\lvert v_{i}(nT)-v_{j}(nT)\rvert}.$$
	Then, we can write $$d_{V}(nT)=\langle v_{i}(nT)-v_{j}(nT),v\rangle.$$
	Now, we distinguish two cases.
	\\\par\textit{Case I.} Assume that there exists $t_{0}\in [nT-T-
	\bar\tau,nT]$ such that $$\langle v_{i}(t_0)-v_{j}(t_0),v\rangle<0.$$
	Thus, using \eqref{4dist} with $t=nT$ and $\bar{t}= t_{0}$, we have 
	\begin{equation}\label{firstestimatedist}
		\begin{split}
			d_V(nT)&\leq e^{-K(nT-t_{0})}\langle v_{i}(t_0)-v_{j}(t_0),v\rangle+(1-e^{-K(nT-t_{0})})F_{n-2}\\&<(1-e^{-K(nT-t_{0})})F_{n-2}\leq (1-e^{-K(T+{\bar\tau})})F_{n-2}.
		\end{split}
	\end{equation}
	\par \textit{Case II.}  Assume that \begin{equation}\label{scalposdist}
		\langle v_{i}(t)-v_{j}(t),v\rangle\geq0,\quad \forall t\in [nT-T-\bar{\tau},nT]
	\end{equation}
	We set
	$$M_{n-1}:=\max_{l=1,\dots,N}\max_{r\in [nT-T-{\bar\tau},nT-T]}\langle v_{l}(r),v\rangle,$$
	$$m_{n-1}:=\min_{l=1,\dots,N}\min_{r\in [nT-T-{\bar\tau},nT-T]}\langle v_{l}(r),v\rangle.$$
	By definition, $M_{n-1}-m_{n-1}\leq F_{n-1}$. Also, for a.e. $t\in [nT-T,nT]$, we get
	$$\begin{array}{l}
		\vspace{0.3cm}\displaystyle{\frac{d}{dt}\langle v_{i}(t)-v_{j}(t),v\rangle=\frac{1}{h(t)}\sum_{l:l\neq i}\alpha(t)\int_{t-\tau_{2}(t)}^{t-\tau_1(t)}\beta(t-s)c_{il}(t;s)\langle v_{l}(s)-v_{i}(t),v\rangle ds}\\
		\vspace{0.3cm}\displaystyle{\hspace{2cm}+\frac{1}{h(t)}\sum_{l:l\neq j}\alpha(t)\int_{t-\tau_{2}(t)}^{t-\tau_1(t)}\beta(t-s)c_{jl}(t;s)\langle v_{j}(s)-v_{l}(t),v\rangle ds}\\
		\vspace{0.3cm}\displaystyle{\hspace{1cm}=\frac{1}{h(t)}\sum_{l:l\neq i}\alpha(t)\int_{t-\tau_{2}(t)}^{t-\tau_1(t)}\beta(t-s)c_{il}(t;s)(\langle v_{l}(s),v\rangle-M_{n-1}+M_{n-1}-\langle v_{i}(t),v\rangle) ds}\\
		\vspace{0.3cm}\displaystyle{\hspace{2cm}+\frac{1}{h(t)}\sum_{l:l\neq j}\alpha(t)\int_{t-\tau_{2}(t)}^{t-\tau_1(t)}\beta(t-s)c_{jl}(t;s)(\langle v_{j}(s),v\rangle-m_{n-1}+m_{n-1}-\langle v_{l}(t),v\rangle) ds}\\
		\displaystyle{\hspace{1cm}:=S_{1}+S_{2}.}
	\end{array}$$
Note that, for $t\in [nT-T,nT]$, we have $s\in [nT-T-\bar{\tau},nT]$, for all $s\in [t-\tau_{2}(t),t-\tau_1(t)]$, from which $$m_{n-1}\leq \langle v_{k}(s),v\rangle\leq M_{n-1},\quad \forall s\in [t-\tau_{2}(t),t-\tau_1(t)],\,\,\forall k=1,\dots,N.$$ Combining this last fact with \eqref{lowerbound}, for a.e. $t\in [nT-T,nT]$, being $\alpha(t)\leq 1$, it comes that
	$$\begin{array}{l}
		\vspace{0.3cm}\displaystyle{S_{1}\leq K(M_{n-1}-\langle v_i(t),v\rangle)+\frac{\eta(t)}{h(t)}\frac{\alpha(t)}{N-1}\sum_{l:l\neq i}\int_{t-\tau_{2}(t)}^{t-\tau_1(t)}\beta(t-s)(\langle v_{l}(s),v\rangle-M_{n-1}),}
	\end{array}$$
	and
	$$\begin{array}{l}
		\vspace{0.3cm}\displaystyle{S_{2}\leq K(\langle v_j(t),v\rangle-m_{n-1})+\frac{\eta(t)}{h(t)}\frac{\alpha(t)}{N-1}\sum_{l:l\neq j}\int_{t-\tau_{2}(t)}^{t-\tau_1(t)}\beta(t-s)(m_{n-1}-\langle v_{l}(s),v\rangle).}
	\end{array}$$
	Hence, we get
	$$\begin{array}{l}
		\vspace{0.3cm}\displaystyle{S_{1}+S_{2}\leq K(M_{n-1}-m_{n-1}-\langle v_{i}(t)-v_j(t),v\rangle)}\\
		\vspace{0.3cm}\displaystyle{\hspace{2.5cm}+\frac{\eta(t)}{h(t)}\frac{\alpha(t)}{N-1}\sum_{l:l\neq i,j}\int_{t-\tau_{2}(t)}^{t-\tau_1(t)}\beta(t-s)(\langle v_{l}(s),v\rangle-M_{n-1}+m_{n-1}-\langle v_{l}(s),v\rangle)}\\
		\vspace{0.3cm}\displaystyle{\hspace{2.5cm}-\frac{\eta(t)}{h(t)}\frac{\alpha(t)}{N-1}\int_{t-\tau_{2}(t)}^{t-\tau_1(t)}\beta(t-s)(\langle v_{i}(s)-v_j(s),v\rangle-M_{n-1}+m_{n-1})}\\
		\vspace{0.3cm}\displaystyle{\hspace{1.4cm}=(K-\eta(t)\alpha(t))(M_{n-1}-m_{n-1})-K\langle v_i(t)-v_j(t),v\rangle}\\
		\displaystyle{\hspace{2.5cm}-\frac{\eta(t)}{h(t)}\frac{\alpha(t)}{N-1}\int_{t-\tau_{2}(t)}^{t-\tau_1(t)}\beta(t-s)\langle v_{i}(s)-v_j(s),v\rangle.}
	\end{array}$$
	Now, for $t\in [nT-T,nT]$ it holds $s\in [nT-\bar{\tau},nT]$, for all $s\in [t-\tau_2(t),t-\tau_1(t)]$. Then, by virtue of \eqref{scalposdist}, $\langle v_i(s)-v_j(s),v\rangle\geq 0$, for all $s\in [t-\tau_2(t),t-\tau_1(t)]$. This fact together with $M_{n-1}-m_{n-1}\leq F_{n-1}$ implies 
	$$\begin{array}{l}
		\displaystyle{\frac{d}{dt}\langle v_i(t)-v_j(t),v\rangle\leq (K-\eta(t)\alpha(t))F_{n-1}-K\langle v_i(t)-v_j(t),v\rangle,}
	\end{array}$$
	for a.e $t\in [nT-T,nT]$. Therefore, using Gronwall's inequality with $\langle v_i(nT-T)-v_j(nT-T),v\rangle\leq M-m\leq F_{n-1},$ we can write
	$$\begin{array}{l}
		\displaystyle{\langle v_i(t)-v_j(t),v\rangle\leq \left(1-e^{-KT}\int_{nT-T}^{t}\eta(s)\alpha(s)ds\right)F_{n-1}
			,}
	\end{array}$$
	for every $t\in [nT-T,nT]$. In particular, for $t=nT$, we obtain
	$$\begin{array}{l}
		\displaystyle{d_{V}(nT)\leq \left(1-e^{-KT}\int_{nT-T}^{nT}\eta(s)\alpha(s)ds\right)F_{n-1}.}
	\end{array}$$
	Thus, using the Persistence Excitation condition {\bf (PE)} we find
	\begin{equation}\label{secondestimatedist}
		d_{V}(nT)\leq  \left(1-e^{-KT}\eta(nT)\tilde{\alpha}\right)F_{n-1},
	\end{equation}
	where in the above estimate we used the fact that $\eta(\cdot)$ is a nonincreasing function. Combining \eqref{firstestimatedist} and \eqref{secondestimatedist}, we can conclude that \eqref{dvdist} holds true.
	\\Finally, arguing as in the proof of Proposition \ref{lemma3}, from \eqref{dvdist} it comes that \eqref{n-2dist} is fulfilled.
	
\end{proof}
Once we have proven estimate \eqref{n-2dist}, we can repeat verbatim the proof of Theorem \ref{uf}, obtaining the exponential flocking for system \eqref{hkd}. Therefore, the following flocking result holds for solutions to \eqref{hkd}.
\begin{thm}
	Assume \eqref{delay_distr}. Let $\psi:\RR\rightarrow\RR$ be a positive, bounded, continuous function that satisfies
	\eqref{infint}. Assume that the weight function $\alpha:[0,+\infty)\rightarrow[0,1]$ is a
	${\mathcal L^1}$-measurable satisfying the Persistence Excitation Condition {\bf (PE)}. Let $\beta:[0,\bar{\tau}]\rightarrow(0,+\infty)$ be a continuous function and let $h:[0,+\infty)\rightarrow (0,+\infty)$ be the function defined in \eqref{h(t)}. Moreover, let $x^{0}_{i},v_{i}^{0}:[-{\bar\tau},0]\rightarrow \RR^{d}$ be continuous functions, for any $i=1,\dots,N$. 
	Then, for every solution $\{(x_{i},v_{i})\}_{i=1,\dots,N}$ to \eqref{hkd} with the initial conditions \eqref{inconddist}, there exists a positive constant $\rho^{*}$ such that \begin{equation}\label{posbounddist}
		\sup_{t\geq-{\bar\tau}}d_{X}(t)\leq \rho^{*},
	\end{equation}
	and there exists another positive constant $\nu$, independent of $N$, for which the following exponential decay estimate holds
	\begin{equation}\label{veldist}
		d_{V}(t)\leq \left(\max_{i,j=1,\dots,N}\,\,\max_{r,s\in [-\bar\tau,0]}\lvert v_{i}(r)-v_{j}(s)\rvert\right) e^{-\nu(t-3T)},\quad \forall t\geq 0,
	\end{equation}
	where $T>0$ is the constant in the Persistence Excitation Condition {\bf (PE)}.	
\end{thm}	

	\bigskip
	\noindent {\bf Acknowledgements.}  The author is member of Gruppo Nazionale per l’Analisi Matematica, la Probabilità e le loro Applicazioni (GNAMPA) of the Istituto Nazionale di Alta Matematica (INdAM). The author is supported by PRIN 2022 (2022W58BJ5) {\em PDEs and optimal control methods in mean field games, population dynamics and multi-agent models}, and by INdAM GNAMPA Project {\em New trends in multiagent systems and mean field games} (CUP E5324001950001).

\end{document}